\documentclass[english,11pt]{article}
\usepackage{latexsym,amssymb,amsmath,amsthm,epsfig,braket,color}
\usepackage{setspace}
\usepackage{comment}

\usepackage[top=25truemm,bottom=25truemm,left=25truemm,right=25truemm]{geometry}

\newtheorem{theorem}{Theorem}[section]
\newtheorem{lemma}[theorem]{Lemma}  
\newtheorem{proposition}[theorem]{Proposition}

\newtheorem{conjecture}[theorem]{Conjecture}

\theoremstyle{definition}
\newtheorem{definition}[theorem]{Definition}
\newtheorem{example}[theorem]{Example}
\newtheorem{remark}[theorem]{Remark}    
\newtheorem{problem}[theorem]{Problem}

\newcommand{\Z}{\mathbf{Z}}
\newcommand{\R}{\mathbf{R}}
\newcommand{\Zp}{\mathbf{Z}_{\geq 0}}

\newcommand{\J}{\mathcal{J}}
\newcommand{\I}{\mathcal{I}}
\newcommand{\B}{\mathcal{B}}
\newcommand{\F}{\mathcal{F}}
\newcommand{\A}{\mathcal{A}}

\newcommand{\change}[1]{\textcolor{black}{#1}}
\newcommand{\changeF}[1]{\textcolor{black}{#1}}
\begin{document}
\title{A Generalized-Polymatroid Approach \\to Disjoint Common Independent Sets in Two Matroids}
\author{
Kenjiro Takazawa%
\thanks{%Department of Industrial and Systems Engineering, Faculty of Science and Engineering, 
Hosei University,
Tokyo 184-8584, Japan.  
E-mail: {\tt takazawa@hosei.ac.jp}. 
}
~and~
Yu Yokoi%
\thanks{National Institute of Informatics, Tokyo 101-8430, Japan. 
E-mail: {\tt yokoi@nii.ac.jp}. 
}
}
%\date{First version: March 23, 2018~~~ This version: \today}
\date{January 2019}
\maketitle

\begin{abstract}
In this paper,
we investigate the classes of matroid intersection admitting a solution for
the problem of partitioning the ground set $E$ into $k$ common independent sets,
where $E$ can be partitioned into $k$ independent sets in each of the
two matroids.
For this problem, we present a new approach building upon the
generalized-polymatroid
intersection theorem. We exhibit that this approach offers alternative
proofs and
unified \change{understanding} of previous results
showing that the problem has a solution
%%for some matroid classes.
for the intersection of two laminar matroids and that of two matroids
without $(k+1)$-spanned elements.
Moreover, we
%%show a new class of matroid intersection for which the problem has a solution.
newly show that the intersection of a laminar matroid and a matroid
without $(k+1)$-spanned elements
admits a solution.
We also construct an example of a transversal matroid which is incompatible with
the generalized-polymatroid approach.
\end{abstract}

\section{Introduction}

For two matroids with a common ground set, 
the problem of 
partitioning the ground set into common independent sets is a classical topic in discrete mathematics. 
That is, 
extending \change{K\H{o}nig's celebrated} bipartite edge-coloring theorem \cite{Konig16}, 
described below, into general matroid intersection has been \change{studied extensively}. 
\changeF{In this paper, we call a family $\mathcal{P}$ of subsets of $E$ a {\em partition} of $E$ if the members of $\mathcal{P}$ are pairwise disjoint and their union  is $E$.
(We allow empty subsets in $\mathcal{P}$.)}
\begin{theorem}[K\H{o}nig \cite{Konig16}]\label{thm:Konig}
For a bipartite graph $G$ and a positive integer $k$, 
the edge set of $G$ 
can be partitioned into $k$ matchings 
if and only if the maximum degree of the vertices of $G$ is at most $k$.
\end{theorem}
Let $G=(U,V;E)$ be a bipartite graph. 
A subset $X$ of $E$ is 
a matching if and only if it is a common independent set
of two partition matroids $M_{1}=(E,\I_{1})$ and $M_{2}=(E,\I_{2})$,
where $\I_{1}$ (resp.,\ $\I_{2}$) is the family of edge sets 
in which no two edges are adjacent at $U$ (resp.,\ at $V$). 
The maximum degree of $G$ coincides with the minimum number $k$ 
such that $E$ can be partitioned into $k$ independent sets of $M_{1}$ and also 
into $k$ independent sets of $M_{2}$.
We then naturally conceive the following problem for a general matroid pair on the common ground set.

\begin{problem}\label{prob:main}
Given two matroids $M_{1}=(E,\I_{1})$ and $M_{2}=(E,\I_{2})$ and a positive integer $k$ 
such that $E$ can be partitioned into $k$ independent sets of $M_{1}$ and also 
into $k$ independent sets of $M_{2}$, find a
partition of $E$ into $k$ common independent sets of $M_{1}$ and $M_{2}$.
\end{problem}

Solving Problem \ref{prob:main} amounts to extending Theorem \ref{thm:Konig} into matroid intersection. 
Such an extension is proved for arborescences in digraphs \cite{Edm73} 
and the intersection of two strongly base orderable matroids \cite{DM76}, 
while such an extension is impossible for a simple example of the intersection of a graphic matroid and a partition matroid 
on the edge set of $K_4$ \cite[Section 42.6c]{Schr03}. 
Indeed, 
%%while 
Problem \ref{prob:main} is 
%%a natural generalization of bipartite edge-coloring, 
%%it is 
known to be a challenging problem:
we only have partial answers in the literature \cite{AB06,DM76,Edm73,HKL11,KZ05}, 
\change{and a sufficient condition for Problem \ref{prob:main} to have a solution is equivalent to 
the famous conjecture of Rota (in \cite{HR94}).}

\begin{conjecture}[Rota's conjecture (see \cite{HR94})]
\change{
Let $M=(E,\I)$ be a matroid 
%%of rank $k$ 
such that $E$ has a partition $\{A_1, A_2\ldots, A_k\}$, 
where $A_i$ is a base of $M$ and $|A_i|=k$ for each $i=1,2,\ldots,k$. 
Then, 
there exists another patition $\{B_1,B_2,\ldots, B_k\}$ of $E$ such that 
each $B_j$ $(j=1,2,\ldots, k)$ is a base of $M$ and $|A_i \cap B_j|=1$ for each $i=1,2,\ldots, k$. }
\end{conjecture}

An interesting class of matroid intersection 
for which Problem~\ref{prob:main} admits a solution is introduced by Kotlar and Ziv \cite{KZ05}. 
For a matroid $M$ on ground set $E$ and a positive integer $k$, 
an element $e$ of $E$ is called \emph{$k$-spanned} if 
there exist $k$ disjoint sets spanning $e$ (see Section \ref{sec:matroid} for definition). 
Kotlar and Ziv \cite{KZ05} presented two sufficient conditions (Theorems \ref{thm:KZ1} and \ref{thm:KZ2})
for the common ground set $E$ of two matroids $M_1$ and $M_2$ 
to be \change{partitionable} into $k$ common independent sets, 
under the assumption that 
%%\changeF{every element of $E$ is $(k+1)$-spanned in neither $M_1$ nor $M_2$}. 
no element of $E$ is $(k+1)$-spanned in $M_1$ or $M_2$. 
Since 
the ground set of 
a matroid in which no element  is $(k+1)$-spanned can be partitioned into $k$ independent sets (Lemma \ref{lem:single_matroid}), 
these two cases offer classes of matroid intersection for which Problem~\ref{prob:main} is solvable.

In this paper, 
we present a new approach to Problem \ref{prob:main} building upon 
the integrality of 
\emph{generalized polymatroids} \cite{Frank84,Has82}, 
a comprehensive class of polyhedra associated with a number of tractable combinatorial structures. 
This generalized-polymatroid approach is regarded as an extension of the polyhedral approach to bipartite edge-coloring \cite[Section 20.3]{Schr03}, 
and 
is indeed successful in \emph{supermodular coloring} \cite{Tardos85}, 
which is another matroidal generalization of bipartite edge-coloring (see \cite{Schr85,Schr03}). 
Utilizing generalized polymatroids, \change{in Section~\ref{sec:main},}
we offer alternative proofs and unified \change{understanding} for some special cases 
for which Problem~\ref{prob:main} admits solutions. 
To be more precise, 
we first prove the extension of Theorem \ref{thm:Konig} 
for the intersection of two \emph{laminar matroids}. 
Laminar matroids 
recently attract particular attention based on \change{their} relation to the matroid secretary problem 
(see \cite{FO17} and references therein), 
and 
form a special case of strongly base orderable matroids. 
Thus, 
the generalized-polymatroid approach yields 
another proof for a special case of \cite{DM76}. 
We then show alternative proofs for the two cases of Kotlar and Ziv \cite{KZ05}, 
which offer a new understanding of the tractability of the two cases. 
Moreover, 
we newly prove that Problem~\ref{prob:main} admits a solution for the intersection of 
a laminar matroid and a matroid in the two classes of Kotlar and Ziv \cite{KZ05}. 
Finally, \change{in Section~\ref{sec:incompatible},} we show a limit of the generalized-polymatroid approach by constructing 
an instance of a transversal matroid, 
another special class of strongly base orderable \change{matroids}, 
which is incompatible with
the generalized-polymatroid approach. 

%The rest of the paper is organized as follows. 
%In Section \ref{sec:pre}, 
%we review the definitions and some fundamental properties of matroids and generalized polymatroids. 
%Section \ref{sec:main} is devoted to solving Problem~\ref{prob:main}
%for some classes of matroid intersection 
%by our generalized-polymatroid approach. 
%In contrast, 
%in Section \ref{sec:incompatible}, 
%we construct a transversal matroid which is incompatible with the generalized-polymatroid approach. 
%Section \ref{sec:concl} concludes the paper. 

\section{Preliminaries}
\label{sec:pre}
In this section, 
we review the definition and 
fundamental properties of matroids and generalized polymatroids. 
For more details, 
the readers are referred to \cite{FT88,Fuj05,Mbook,Oxl11,Schr03,Wel76}. 

\subsection{Matroid}\label{sec:matroid}
Let $E$ be a finite set. 
For a subset $X\subseteq E$ 
and elements $e\in E\setminus X$, $e'\in X$, 
we denote $X+e=X\cup\{e\}$ and $X-e'=X\setminus\{e'\}$.
For a set family $\I\subseteq 2^{E}$, 
a pair $(E, \I)$ is called a {\em matroid} if $\I$ satisfies 
\begin{itemize}
\setlength{\parskip}{0.0cm}
\setlength{\itemsep}{0.0mm}
\item[(I0)] $\emptyset\in \I$, 
\item[(I1)] $X\subseteq Y\in \I$ implies $X\in \I$, and 
\item[(I2)] If $X,Y\in \I$ and $|X|<|Y|$, then $\exists e\in Y\setminus X: X+e\in \I$.
\end{itemize}
Each member $X$ of $\I$ is called an {\em independent set}. 
In particular, 
an independent set $B\in \I$ is called a {\em base} if it is maximal in $\I$ with respect to inclusion. 
It is known that all bases have the same size.

The {\em rank function} $r:2^{E}\to \Zp$ of a matroid $M=(E, \I)$
is defined by $r(A)=\max\{|X|\mid X\subseteq A,~X\in \I\}$ for any $A\subseteq E$.
Then, it is known that $\I=\set{X| \forall A\subseteq E: |X\cap A|\leq r(A)}$ holds.
A subset $X \subseteq E$ \emph{spans} an element $e \in E$ if $r(X + e) = r(X)$.

For a matroid $M=(E,\I)$ and a subset $S\subseteq E$,
the restriction $M|S$ of $M$ to $S$ is a pair $(S, \I|S)$,
where $\I|S=\set{X|X\in \I,~X\subseteq S}$.
For any $S\subseteq E$, the restriction $M|S$ is again a matroid.

For a matroid $M=(E,\I)$ and a positive integer $k\in \Z$, 
we define a set family $\I^{k}\subseteq 2^{E}$ by
\[\I^{k}=\set{X\subseteq E|\text{$X$ can be partitioned into $k$ sets in $\I$}}.\]
\changeF{%
The following theorem is a special case of  
Edmonds' famous matroid partition theorem \cite{Edmonds65}.
\begin{theorem}[Edmonds \cite{Edmonds65}] \label{thm:matroid_union}
For a matroid $M=(E,\I)$ whose rank function is $r:2^{E}\to \Zp$ and a positive integer $k\in \Z$, 
the pair $M^{k}=(E,\I^{k})$ is a matroid and its rank function $r^{k}:2^E\to \Z$ is given by the following formula:
\begin{equation}
\textstyle r^k(X)=\min\{|X\setminus Y|+k\cdot r(Y)\mid Y\subseteq X \}.
\label{eq:truncation3}
\end{equation}
\end{theorem}
Then, $r^{k}(X)$ is the maximum size of a subset of $X$ partitionable into $k$ independent sets.
}

%The rank function of any matroid is known to satisfy 
%\begin{itemize}
%\setlength{\parskip}{0.0cm}
%\setlength{\itemsep}{0.0mm}
%\item[(R0)] $r(\emptyset)=0$, 
%\item[(R1)] $A\subseteq B\subseteq E$ implies $r(A)\leq r(B)$, and 
%\item[(R2)] $r(A)+r(B)\geq r(A\cup B)+r(A\cap B)$ for any $A, B\subseteq E$.
%\end{itemize}

\subsection{Generalized Polymatroid}

Let $E$ be a finite set.
A function $b:2^{E}\to \R\cup\{\infty\}$ is called {\em submodular}
if it satisfies the {\em submodular inequality}
\begin{equation*}
b(A)+b(B)\geq b(A\cup B)+b(B\cap A) 
%%\label{eq:submo-ineq}
\end{equation*}
for any $A, B\subseteq E$, 
where the inequality is \change{assumed} to hold if the left-hand side is infinite.
A function $p:2^{E}\to \R\cup\{-\infty\}$ is called {\em supermodular}
if $-p$ is submodular.
A pair $(p,b)$ is called {\em paramodular} if 
we have 
\begin{itemize}
\setlength{\parskip}{0.0cm}
\setlength{\itemsep}{0.0mm}
\item[(i)] $p(\emptyset)=b(\emptyset)=0$, 
\item[(ii)] $p$ is supermodular, $b$ is submodular, and
\item[(iii)] $p$ and $b$ satisfy the {\em cross inequality}
\begin{equation}
b(A)-p(B)\geq b(A\setminus B)-p(B\setminus A) \label{eq:cross-ineq}
\end{equation}
for any $A,B\subseteq E$, where the inequality is \change{assumed} to hold if the left-hand side is infinite.
\end{itemize}

For a pair of set functions $p:2^{E}\to \R\cup\{-\infty\}$ and $b:2^{E}\to \R\cup \{\infty\}$,
we associate a polyhedron $Q(p,b)$ defined by 
\[Q(p,b)=\set{x\in \R^{E}|\forall A\subseteq E: p(A)\leq x(A)\leq b(A)},\]
where $x(A)=\sum\set{\change{x_e}| e\in A}$. 
Here, $p$ serves as a lower bound while $b$ serves as an upper bound of the polyhedron $Q(p,b)$.
A polyhedron $P\subseteq \R^{E}$ is called a {\em generalized polymatroid} (for short, a {\em g-polymatroid})
if $P=Q(p,b)$ holds for some paramodular pair $(p,b)$.
It is known \cite{Frank84, FT88} that such a paramodular pair is uniquely defined 
for any g-polymatroid. 
%%\footnote{\changeF{Generalized-polymatroid is also known to be equivalent to base polyhedron in the sense that any g-polymatroid can be obtained as 
%%a projection of a base polyhedron \cite[Theorem 3.58]{Fuj05}.}}%
%%.
%% KT edit
\changeF{It is also known that a generalized-polymatroid is equivalent to a base polyhedron: 
any g-polymatroid can be obtained as 
a projection of a base polyhedron \cite{Fuj84} (see also \cite[Theorem 3.58]{Fuj05}).}

%Here $(p,b)$ is called the {\em border pair} of $Q:=Q(p,b)$ while $p$ is the
%{\em lower} and $b$ is the {\em upper border function} of $Q$. 
%If, in addition, $(p,b)$ is integral, $p$ is nonnegative and $b$ is subcardinal, 
%then we say that the pair $(p,b)$ is {\em g-matroidal}. 
%and that the pair (E, F (p,b)) is a generalized matroid (for short, a g-matroid).

%%\medskip
We next introduce the concept of {\em intersecting paramodularity}, which is
weaker than paramodularity but still yields g-polymatroids. 
We say that subsets $A,B\subseteq E$ are {\em intersecting} 
if none of $A\cap B$, $A\setminus B$ and $B\setminus A$ is empty.
A function $b:2^{E}\to \R\cup\{\infty\}$ is called {\em intersecting submodular} if
it satisfies the submodular inequality for any intersecting subsets $A, B\subseteq E$.
A function $p:2^{E}\to \R\cup\{-\infty\}$ is called {\em intersecting supermodular}
if $-p$ is intersecting submodular.
A pair $(p,b)$ is called {\em intersecting paramodular} if $p$ and $b$ are
intersecting super- and submodular functions, respectively, and
the cross inequality \eqref{eq:cross-ineq} holds for any intersecting subsets $A, B\subseteq E$.

The following theorem, 
\changeF{which is derived from the fact that 
intersecting paramodularity corresponds to the projection of \emph{crossing submodularity} \cite{Fuj84,Fuj05}}, 
%%gives a sufficient condition 
%%for a function pair to 
states that an intersecting-paramodular pair $(p,b)$
defines a g-polymatroid.
We say that a pair $(p,b)$ of set functions 
is {\em integral} if each of $b(A)$ and $p(A)$ is \change{an} integer or infinite for any $A\subseteq E$.
We say that a polyhedron is {\em integral} if each of its \change{nonempty} faces contains an integral point,
%It suffices to require integrality only for minimal faces, 
so for pointed polyhedra, the vertices should be integral.
\begin{theorem}[Frank \cite{Frank84}]\label{thm:intersecting}
For an intersecting-paramodular pair $(p,b)$ such that $Q(p,b)\neq \emptyset$,
the polyhedron $Q(p,b)$ is a g-polymatroid, which is, 
in addition, integral whenever $(p,b)$ is integral. 
\end{theorem}

In general, 
the intersection of two integral polyhedra $P_1$ and $P_2$ 
is not necessarily integral. 
For two integral g-polymatroids, 
however, 
the intersection 
preserves integrality as stated below. 
This fact plays a key role in our g-polymatroid approach to Problem 1.2. 

\begin{theorem}[Integrality of g-polymatroid intersection \cite{Frank84}]\label{thm:intersection}
For two integral g-polymatroids $P_{1}$ and $P_{2}$, 
the intersection $P_{1}\cap P_{2}$ is an integral polyhedron if it is nonempty.	
\end{theorem}

As this paper studies partitions of finite sets,
we are especially interested in vectors in 
the intersection of a g-polymatroid and the unit hypercube $[0,1]^{E}=\set{x\in \R^{E}|\forall e\in E:0\leq x(e)\leq 1}$.
It is known that the intersection is again a g-polymatroid.
\begin{theorem}[Frank \cite{Frank84}]\label{thm:unit_cube}
For a g-polymatroid $P$, if $P\cap [0,1]^{E}$ is nonempty, then
the intersection $P\cap [0,1]^{E}$ is again a g-polymatroid, which is,
in addition, integral whenever $P$ is integral. 
\end{theorem}
Similarly to the definition of $Q(p,b)$,
for a pair of set functions $p:2^{E}\to \R\cup\{-\infty\}$ and $b:2^{E}\to \R\cup \{\infty\}$,
we associate the following set family:
\[\F(p,b)=\set{X\subseteq E|\forall A\subseteq E: p(A)\leq |X\cap A|\leq b(A)}.\]

For a subset $Y\subseteq E$, its {\em characteristic vector} $\chi_{Y}\in \{0,1\}^{E}$
is defined by $\chi_{Y}(e)=1$ for $e\in Y$ and $\chi_{Y}(e)=0$ for $e\in E\setminus Y$. 
The following observation is derived from Theorem~\ref{thm:unit_cube}.
\begin{lemma}\label{lem:convex}
For an integral intersecting-paramodular pair $(p,b)$,
the polyhedron $Q(p,b)\cap [0,1]^{E}$ is a convex hull of 
the characteristic vectors of the members of $\F(p,b)$.
\end{lemma}
\begin{proof}
By Theorem~\ref{thm:unit_cube}, $Q(p,b)\cap [0,1]^{E}$ is integral,
and hence all its vertices are $(0, 1)$-vectors. 
Also, 
%%by definition, 
by the definition of $Q(p,b)$ and $\F (p,b)$, 
we have $y\in Q(p,b)\cap \{0,1\}^{E}$ if and only if $y=\chi_{Y}$ for some $Y\in \F(p,b)$.
Thus, the vertices of $Q(p,b)\cap [0,1]^{E}$ coincides with the
characteristic vectors of the members of $\F(p,b)$.
\end{proof}

\begin{comment}
Combining Theorems~\ref{thm:intersecting}--\ref{thm:intersection} immediately implies the following.
%
\begin{lemma}\label{lem:intersection_family_var}
For two intersecting-paramodular pairs $(p_{1},b_{1})$ and $(p_{2},b_{2})$ that are integral,
if the intersection $Q(p_{1},b_{1})\cap Q(p_{2},b_{2})\cap [0,1]^{E}$ is nonempty, 
then there exists a subset $X\subseteq E$
such that $X\in \F(p_{1},b_{1})\cap \F(p_{2},b_{2})$.
\end{lemma}
\begin{proof}
%Since $Q(p_{1},b_{1})\cap Q(p_{2},b_{2})\cap [0,1]^{E}$ is nonempty,
%the polyhedra $Q(p_{i},b_{i})$ and $Q(p_{i},b_{i})\cap [0,1]^{E}$ are nonempty for $i=1,2$.
By Theorem~\ref{thm:intersecting}, $Q(p_{i},b_{i})$ is an integral g-polymatroid,
and so is $Q(p_{i},b_{i})\cap [0,1]^{E}$ by Theorem~\ref{thm:unit_cube}.
By Theorem~\ref{thm:intersection}, the intersection 
$Q(p_{1},b_{1})\cap Q(p_{2},b_{2})\cap [0,1]^{E}$ is integral, and hence it contains a $(0,1)$-vector, say $x$.
Then there exists $X\in \F(p_{1},b_{1})\cap\F(p_{2},b_{2})$ such that $\chi_{X}=x$.
%Take such a vector $x\in \{0,1\}^{E}$ and let $X=\set{e\in E|x(e)=1}$. 
%Then $x\in Q(p_{i},b_{i})$ implies $X\in \F(p_{i},b_{i})$ for each $i=1,2$.
\end{proof}
\end{comment}

\section{Generalized-Polymatroid Approach}
\label{sec:main}

In this section, 
we exhibit 
some cases of matroid intersection for which a solution of Problem~\ref{prob:main} can be constructed 
by utilizing the g-polymatroid intersection theorem (Theorem \ref{thm:intersection}). 
In Section~\ref{sec:general}, 
we describe a general method to apply Theorem \ref{thm:intersection} for solving Problem~\ref{prob:main}. 
%%This method is regarded as an extension of 
%%the polyhedral approach to Theorem \ref{thm:Konig} (\cite[Section 20.3]{Schr03}) 
%%and 
%%a variant of 
%%the generalized-polymatroid approach to 
%%supermodular coloring \cite{Tardos85}. 
In Section~\ref{sec:laminar}, 
we use this method to prove 
an extension of Theorem \ref{thm:Konig} to the intersection of two laminar matroids, 
a special case of intersection of two strongly base orderable matroids \cite{DM76}. 
In Section \ref{sec:k+1-spanned}, 
we utilize this method for alternative proofs for two classes of 
matroid intersection due to Kotlar and Ziv~\cite{KZ05}. 
Section \ref{sec:new} presents a new class of matroid intersection for which Problem~\ref{prob:main} admits a solution: 
intersection of a laminar matroid and a matroid in Kotlar and Ziv's classes. 
Finally, in Section~\ref{sec:complexity}, 
we explain an algorithmic implementation of our general method and analyze its time complexity.

\subsection{General Method}
\label{sec:general}

With the notations introduced in Section~\ref{sec:matroid}, 
now Problem~\ref{prob:main} is reformulated as follows.
%Recall that $M^{k}=(E,\I^{k})$ is the sum of $k$ identical matroid $M=(E,\I)$.

\let\temp\thetheorem
\renewcommand{\thetheorem}{\ref{prob:main}}
\begin{problem}[reformulated]
Given matroids $M_{1}=(E,\I_{1})$ and $M_{2}=(E,\I_{2})$ and a positive integer $k$ 
such that $E\in \I_{1}^{k}\cap \I_{2}^{k}$, find a
partition $\{X_{1},X_{2},\dots,X_{k}\}$ of $E$ such that $X_{j}\in \I_{1}\cap \I_{2}$ for $j=1,2,\dots,k$.
\end{problem}
\let\thetheorem\temp
\addtocounter{theorem}{-1}

Our general method to solve Problem \ref{prob:main} is to 
%%apply Proposition~\ref{prop:g-matroid_approach} repeatedly. 
%%More precisely, for $i=1,2,\dots,k$, we 
find $X \in \I_{1}\cap \I_{2}$ such that $E\setminus X\in \I^{k-1}_{1}\cap \I^{k-1}_{2}$ with the aid of g-polymatroid intersection, 
replace $E$ and $k$ with $E\setminus X$ \change{and} $k-1$, respectively, 
and iterate. 
%%In order to find such $X\subseteq E$ with the aid of g-polymatroid intersection, 
%%for each $i=1,2$, 
%%we need 
%%%%an intersecting supermodular function $p\colon 2^E \to \R \cup \{-\infty\}$ 
%%%%and an intersecting submodular function $b\colon 2^E \to \R \cup \{\infty\}$ be an intersecting submodular function 
%%an intersecting-paramodular pair $(p_i,b_i)$
%%satisfying that 
%%\begin{align*}
%%%%\label{EQpb}
%%\F(p_i,b_i)=\set{X\subseteq E|X\in \I_i,~E\setminus X\in \I_i^{k-1}}.
%%\end{align*}
%%For the class of matroid intersection in which such intersecting-paramodular pairs $(p_1,b_1)$ and $(p_2,b_2)$ exist, 
%%our method is confirmed by the following proposition, 
%%which can be proved by combining Theorems~\ref{thm:intersecting}, \ref{thm:unit_cube} and \ref{thm:intersection}. 
The following proposition, 
which can be proved by combining Theorems~\ref{thm:intersecting}--\ref{thm:unit_cube} and Lemma \ref{lem:convex}, 
shows 
%%the class of matroids 
a necessary condition 
%%of matroids for this method to be applied.  
that this method can be applied. 

%%By combining 
%%%%Theorems~\ref{thm:intersecting}--\ref{thm:intersection} and Lemma \ref{lem:fractional},
%%\tedit{Lemma \ref{lem:fractional} and Theorems~\ref{thm:intersecting}, \ref{thm:unit_cube} and \ref{thm:intersection},}
%%we can obtain the following \tedit{proposition}, which is a core of our g-polymatroid approach.
%%It provides a sufficient condition for an instance of Problem~\ref{prob:main}
%%to be reduced into a smaller instance with $k$ replaced by $k-1$.
%
\begin{proposition}\label{prop:g-matroid_approach}
Let $M_{1}=(E,\I_{1})$, $M_{2}=(E,\I_{2})$ be matroids 
and $k\in \Z$ be a positive integer with $E\in \I_{1}^{k}\cap \I_{2}^{k}$. 
If there exists
an integral intersecting-paramodular \change{pair} $(p_{i},b_{i})$ 
such that 
%%\[\F(p_{i},b_{i})=\set{X\subseteq E|X\in \I_{i},~E\setminus X\in \I_{i}^{k-1}}\] 
\begin{align}
\label{EQpb}
\F(p_i,b_i)=\set{X\subseteq E|X\in \I_i,~E\setminus X\in \I_i^{k-1}} 
\end{align}
for each $i=1,2$, 
then there exists a subset $X\subseteq E$ such that 
$X\in \I_{1}\cap \I_{2}$ and $E\setminus X\in \I_{1}^{k-1}\cap \I_{2}^{k-1}$.
\end{proposition}

\begin{proof}
%%We first show that $Q(p_i,b_i) \cap [0,1]^E$ is nonempty for each $i=1,2$. 
Because $E\in \I_i^{k}$, 
there is a partition $\{X_{1}, X_{2},\dots,X_{k}\}$ of $E$
such that $X_{j}\in \I_i$ for each $j=1,2,\dots,k$.
For each $j$, we have $X_{j}\in \I_i$ and $E\setminus X_{j}=\bigcup_{\ell:\ell\neq j}X_{\ell}\in \I^{k-1}$,
and hence $X_{j}\in \F(p_i,b_i)$. 
As $\{X_{1}, X_{2},\dots,X_{k}\}$ is a partition of $E$,
the vector $x:=\left(\frac{1}{k}, \frac{1}{k},\dots,\frac{1}{k}\right)^{\top}\in \R^E$ coincides with $\sum_{j=1}^{k}\frac{1}{k}\cdot \chi_{X_{j}}$,
which is a convex combination of the characteristic vectors of $X_{j}\in \F(p_i,b_i)$ ($j=1,2,\ldots,k$).
Then, Lemma~\ref{lem:convex} implies $x\in Q(p_i,b_i)\cap [0,1]^{E}$. 

Now $Q(p_{1},b_{1})\cap Q(p_{2},b_{2})\cap [0,1]^{E}$ includes 
the vector $\left(\frac{1}{k}, \frac{1}{k},\dots, \frac{1}{k}\right)^{\top}$,
and hence is nonempty. 
Then, 
by 
combining Theorems~\ref{thm:intersecting}--\ref{thm:unit_cube}, 
we obtain that $Q(p_{1},b_{1})\cap Q(p_{2},b_{2})\cap [0,1]^{E}$ is an integral nonempty polyhedron, 
and hence it contains a $(0,1)$-vector 
%%$y \in Q(p_{1},b_{1})\cap Q(p_{2},b_{2})\cap \{0,1\}^{E}$.
$y$. 
%%Let $X\subseteq E$ be such that $\chi_{X}=x$. 
Let $Y\subseteq E$ be the set satisfying $\chi_{Y}=y$. 
Then $y\in Q(p_{1},b_{1})\cap Q(p_{2},b_{2})\cap [0,1]^{E}$ implies $Y\in \F(p_{1},b_{1})\cap\F(p_{2},b_{2})$, 
which means $Y\in \I_{1}\cap \I_{2}$ and $E\setminus Y\in \I_{1}^{k-1}\cap \I_{2}^{k-1}$.
\end{proof}

In order to use our method, 
$M_1$ and $M_2$ should belong to a class of matroids 
in which each member $M=(E,\I)$ with $E\in \I^{k}$ admits an integral intersecting-paramodular pair
$(p,b)$ satisfying 
%%\tedit{the condition} in Proposition~\ref{prop:g-matroid_approach}
\eqref{EQpb}
and the restriction $M|(E\setminus X)$ with any $X\in \F(p,b)$ belongs to this class again with $k$ replaced by $k-1$. 
In the subsequent subsections, 
we show that the class of laminar matroids 
and the two matroid classes in \cite{KZ05} have this property.

\begin{remark}
\label{REMindependent}
An advantage of our approach 
is that there is no constraint \change{linking} the two matroids $M_1$ and $M_2$. 
%The condition for each matroid is described without using the other.
In other words, 
our approach can deal with any pair of matroids such that each of them admits an 
intersecting-paramodular pair required in Proposition~\ref{prop:g-matroid_approach}.
This \change{contrasts} some previous works \cite{DM76, KZ05}, 
which assume that the two matroids are in the same matroid class.
Indeed, utilizing this fact, we provide \change{results} (Theorems~\ref{thm:lam+KZ1} and \ref{thm:lam+KZ2}) 
that is not included in previous works.
\end{remark}

\subsection{Intersection of Two Laminar Matroids}\label{sec:laminar}
In this section, 
we prove that Problem~\ref{prob:main} is solvable for laminar matroids
by our generalized-polymatroid approach. 
Since a laminar matroid is a generalization of a partition matroid,
this extends the bipartite edge-coloring theorem of K\H{o}nig \cite{Konig16}.
On the other hand, since a laminar matroid is strongly base orderable, 
this proof amounts to another proof for a special case of strongly base orderable matroids by Davies and McDiarmid \cite{DM76}.

We first define the concept of laminar matroids. 
A subset family $\A$ of a finite set $E$ is called \emph{laminar} 
if $A_1,A_2 \in \A$ implies 
$A_1 \subseteq A_2$, 
$A_2 \subseteq A_1$, 
or $A_1 \cap A_2 = \emptyset$. 
Let $\A \subseteq 2^E$ be a laminar family 
and $q:\A\to \Zp$ be a capacity function.
Let $\I$ be a family of subsets $X$ satisfying all capacity constraints, i.e.,
\[\I=\set{X\subseteq E|\forall A\in \A: |X\cap A|\leq q(A)}.\]
Then it is known that $(E,\I)$ is a matroid, 
which we call the {\em laminar matroid} induced from \mbox{$\A$ and $q$}.

%%\begin{definition}[Laminar matroid]
%%A matroid is {\em laminar} if it is induced from some laminar family $\A$ and
%%a capacity function $q$ on $\A$.
%%\end{definition}

It is known that a laminar matroid is a special case of 
a \emph{strongly base orderable matroid} \cite{Bru70}. 

\begin{definition}[Strongly base orderable matroid \cite{Bru70}]
A matroid is {\em strongly base orderable} if for each pair of bases $B_{1}$, $B_{2}$
there exists a bijection $\pi:B_{1}\to B_{2}$ such that for each subset $X$ of $B_{1}$
the set $\pi(X)\cup (B_{1}\setminus X)$ is a base again.
\end{definition}

Thus, 
it follows from the result of Davies and McDiarmid \cite{DM76} 
that  Theorem \ref{thm:Konig} can be extended to 
the intersection of laminar matroids. 
\begin{theorem}[Davies and McDiarmid \cite{DM76}]\label{thm:laminar-partition}
For laminar matroids $M_{1}=(E,\I_{1})$ and $M_{2}=(E,\I_{2})$ and a positive integer $k$ 
such that $E\in \I_{1}^{k}\cap \I_{2}^{k}$, there exists a
partition $\{X_{1},X_{2},\dots,X_{k}\}$ of $E$ such that $X_{j}\in \I_{1}\cap \I_{2}$ for each $j=1,2,\dots,k$.
\end{theorem}
In the rest of this subsection, 
we present an alternative proof for this theorem via the generalized-polymatroid approach. 
We first observe some properties of laminar matroids. 
It is known and can be easily observed that 
the class of laminar matroids is closed under taking restrictions.

\begin{lemma}\label{lem:restriction-laminar-matroid}
Let $M=(E,\I)$ be a laminar matroid induced from a laminar family $\A$ and 
a capacity function $q:\A\to \Zp$. Then for any subset $S\subseteq E$,
the restriction $M|S$ of $M$ to $S$ is 
a laminar matroid induced from 
%%a laminar family $\A_{S}:=\set{A\cap S \mid A\in \A}$
a laminar family \changeF{$\A_{S}:=\set{S' \subseteq S\mid \mbox{$S'=A\cap S$ for some $A\in \A$}}$}
and 
%%the same capacity function $q:\A\to \Zp$.
a capacity function $q_S \colon \A_S \to \Zp$ defined by 
%%$q_S(A \cap S) = q(A)$ for $S \cap A \in \A_S$. 
\changeF{$q_S(S') = \min\set{q(A) \mid A \in \A,~S \cap A = S'}$}. 
\end{lemma}

The next lemma%
\footnote{\changeF{As will be mentioned in Remark~\ref{rem:general-case}, Lemma~\ref{lem:sum-of-laminar-matroid} can extend to a more general case, where a matroid is defined by an intersecting-submodular function. By restricting to laminar matroids, here we provide an elementary and self-contained proof.}}
states that, 
if $M=(E,\I)$ is a laminar matroid induced from a laminar family $\A$, 
then $M^k = (E, \I^k)$ is also a laminar matroid induced from $\A$.

\begin{lemma}\label{lem:sum-of-laminar-matroid}
Let $M=(E,\I)$ be a laminar matroid induced from a laminar family $\A$ and 
a capacity function $q:\A\to \Zp$. Then for a positive integer $k$,
 the matroid $M^{k}=(E,\I^{k})$ is a laminar matroid defined by
\[\I^{k}=\set{X\subseteq E|\forall A\in \A: |X\cap A|\leq k\cdot q(A)}.\]
\end{lemma}
\begin{proof}
We show that, for any $X\subseteq E$,
there exists a partition $\{Y_{1},Y_{2},\dots,Y_{k}\}$ of $X$ with $Y_{j}\in \I~(j=1,\dots,k)$
if and only if $|X\cap A|\leq k\cdot q(A)$ for any $A\in \A$.
The necessity is clear, because each $Y_{j}$ satisfies $|Y_{j}\cap A|\leq q(A)$ for any $A\in \A$.
For the sufficiency, suppose $|X\cap A|\leq k\cdot q(A)$ for any $A\in \A$.
Let $X=\{e_{1},e_{2},\dots,e_{|X|}\}$ (i.e., give indices for the elements in $X$),
so that for all $A\in \A$ the elements in $X\cap A$ have consecutive indices.
This can be done easily because $\A$ is a laminar family%
\footnote{Let $\A_{X}=\{X\}\cup \set{X\cap A|A\in \A}\cup \set{\{e\}|e\in X}$. 
Since $\A$ is laminar, $\A_{X}$ is also laminar.
Let $T$ be a tree representation of $\A_{X}$, i.e., the node sets of $T$ is $\A_{X}$ and 
a node $A$ is a child of $A'$ if $A\subsetneq A'$ and there is no $A''$ with $A\subsetneq A''\subsetneq A'$.
Then each leaf is the singleton of an element in $X$.
Let $X=\{e_{1},e_{2},\dots,e_{|X|}\}$ so that the indices represent 
the order in which the corresponding leaves are found in depth-first search from the root node $X$.
These indices satisfy the required condition.}.
For each $j\in\{1,2,\dots,k\}$, 
let \mbox{$Y_{j}=\set{e_{\ell}\in X \mid \ell = j \mod k}$}.
Then, $\{Y_{1},Y_{2},\dots,Y_{k}\}$ is a partition of $X$, and,
for each $Y_{j}$ and $A\in \A$, we have
$|Y_{j}\cap A|\leq \lceil |X\cap A|/k \rceil$ by the definition of the indices.
Because $|X\cap A|\leq k\cdot q(A)$, this implies $|Y_{j}\cap A|\leq q(A)$ for all $A\in \A$.
Thus, we have $Y_{j}\in \I$ for each $j\in \{1,2,\dots,k\}$.
\end{proof}

The next lemma provides an integral intersecting-paramodular pair $(p,b)$ satisfying the condition in Proposition \ref{prop:g-matroid_approach} 
for a laminar matroid. 

\begin{lemma}\label{lem:g-polymatroid-approach-laminar}
Let $M=(E,\I)$ be a matroid induced from a laminar family $\A$ and 
a function $q:\A\to \Zp$ and suppose $E\in \I^{k}$ for a positive integer $k$.
Define $p:2^{E}\to \Z\cup\{-\infty\}$ and $b:2^{E}\to \Z\cup \{\infty\}$ by
\begin{alignat*}{2}
&p(A)=|A|-(k-1)\cdot q(A)\qquad &(A\in \A),\\
&b(A)=q(A)  &(A\in \A),
\end{alignat*}
where $p(B)=-\infty$, $b(B)=\infty$ for all $B\in 2^{E}\setminus\A$.
Then $(p,b)$ is an integral intersecting-paramodular pair satisfying 
$\F(p,b)=\set{X\subseteq E|X\in \I,~E\setminus X\in \I^{k-1}}$.
\end{lemma}
\begin{proof}
Since $\A$ is laminar and the values of $p$ and $b$ 
are finite only on $\A$, there is no intersecting pair of subsets of $E$
both of which have finite function values.
Thus, $(p,b)$ is trivially intersecting paramodular. 

For any $X\subseteq E$, 
the condition 
$\forall A\in A: |X\cap A|\geq p(A)=|A|-(k-1)\cdot q(A)$ is equivalent to 
$\forall A\in \A: |(E\setminus X)\cap A|\leq (k-1)\cdot q(A)$,
and hence equivalent to $E\setminus X\in \I^{k-1}$ by Lemma~\ref{lem:sum-of-laminar-matroid}.
Also, $\forall A\in \A: |X\cap A|\leq b(A)=q(A)$ is equivalent to $X\in \I$. 
Thus we have $X\in \F(p,b)$ if and only if $X\in \I$ and $E\setminus X\in \I^{k-1}$ hold. 
\end{proof}
Now we show that Problem~\ref{prob:main} can be solved 
for any pair of laminar matroids using the generalized-polymatroid approach.
\begin{proof}[Proof of Theorem~\ref{thm:laminar-partition}]
We show the theorem by induction on $k$.
The case $k=1$ is trivial. Let $k\geq 2$ and suppose that the statement holds for $k-1$.
By Lemma~\ref{lem:g-polymatroid-approach-laminar}, for each $i=1,2$, there exists
an integral intersecting-paramodular pair $(p_{i},b_{i})$ such that
$\F(p_{i},b_{i})=\set{X\subseteq E|X\in \I_{i},~E\setminus X\in \I_{i}^{k-1}}$.
Then, by Proposition~\ref{prop:g-matroid_approach}, there exists 
$X\in \I_{1}\cap \I_{2}$ satisfying $E\setminus X\in \I_{1}^{k-1}\cap \I_{2}^{k-1}$.
By Lemma~\ref{lem:restriction-laminar-matroid},
the restrictions $M'_{1}:=M_{1}|(E\setminus X)$ and $M'_{2}:=M_{2}|(E\setminus X)$ are laminar.
Therefore, by the induction hypothesis,
$E\setminus X$ can be partitioned into $k-1$ common independent sets of
$M'_{1}$ and $M'_{2}$, and hence of $M_{1}$ and $M_{2}$.
Thus, $E$ can be partitioned into $k$ common independent sets.
\end{proof}

%\noindent\changeF{\bf NEW PART $\downarrow$}
\begin{remark}\label{rem:general-case}
Here we mention an extension of Lemma~\ref{lem:sum-of-laminar-matroid}.
For an intersecting-submodular function $b:2^{E}\to \Zp\cup\{\infty\}$, 
define a family $\I_b=\set{X\subseteq E|\forall A\subseteq E: |X\cap A|\leq b(A)}$. 
Then, it is known \cite{Edmonds70} that $(E,\I_b)$ is a matroid. 
Actually, a laminar matroid is a special case of such matroids: 
When $(E,\I)$ is a laminar matroid induced by a laminar family $\A$ and a capacity function $q:\A\to \Zp$, 
then 
$\I = \I_b$ holds for an intersecting-submodular function $b$ defined by $b(A)=q(A)$ for $A\in \A$ and $b(A)=\infty$ for $A\in 2^E\setminus \A$.
Lemma~\ref{lem:sum-of-laminar-matroid} can extends to this matroid class.
That is, for any intersecting-submodular function $b$, 
the family $\I_b^k$ %of subsets partitionable into $k$ sets in $\I_b$ 
can be represented as $\I_b^k=\set{X\subseteq E|\forall A\subseteq E: |X\cap A|\leq k\cdot b(A)}$.

Now we show this claim. As shown by Edmonds \cite{Edmonds70}, the rank function of $(E,\I_b)$ is given as
\begin{equation}
\textstyle r_b(X)=\min\{|X\setminus(Y_1\cup Y_2\cup \cdots \cup Y_l)|+\sum_{i=1}^{l}b(Y_i)\mid Y_1, Y_2,\dots,Y_l \text{ are pairwise disjoint} \}.
\label{eq:truncation1}
\end{equation}
Note that $k\cdot b$ is also an intersecting-submodular function on $E$. Hence, it defines a matroid $(E, \I_{k\cdot b})$ where 
$\I_{k\cdot b}:=\set{X\subseteq E|\forall A\subseteq E: |X\cap A|\leq k\cdot b(A)}$, 
and its rank function is given as 
\begin{equation*}
\textstyle r_{k\cdot b}(X)=\min\{|X\setminus(Y_1\cup Y_2\cup\cdots \cup Y_l)|+\sum_{i=1}^{l}k\cdot b(Y_i)\mid Y_1, Y_2,\dots,Y_l \text{ are pairwise disjoint} \}.
\label{eq:truncation2}
\end{equation*}
On the other hand, by Theorem~\ref{thm:matroid_union}, the rank function $r_b^k$ of the matroid $(E, \I_b^{k})$, is given by \eqref{eq:truncation3} using $r_b$.
Substituting \eqref{eq:truncation1} to \eqref{eq:truncation3}, we can check that $r_b^k=r_{k\cdot b}$. Thus, $\I^k_b=\I_{k\cdot b}$ is proved.

We remark that, even Lemma~\ref{lem:sum-of-laminar-matroid} extends to this matroid class,
our approach for Problem~\ref{prob:main} does not extend because Lemma~\ref{lem:g-polymatroid-approach-laminar} fails to extend to this class.
\end{remark}

%%The above proof of Theorem~\ref{thm:laminar-partition} shows that
%%the induction method in Section~\ref{sec:general} can be applied to 
%%the intersection of two laminar matroids successfully, 
%%which gives an alternative proof for the extension of Theorem~\ref{thm:Konig} to laminar matroids. 

\subsection{Intersection of Two Matroids without $(k+1)$-Spanned Elements}
%%\subsection{Intersection of Kotlar--Ziv Matroids}
\label{sec:k+1-spanned}
%%To provide solvable special cases of Problem~\ref{prob:main}, 
%%Kotlar and Ziv \cite{KZ05} introduced the notion of $k$-spanned elements. 
Let $M=(E,\I)$ be a matroid and $k$ be a positive integer. 
Recall that 
an element $e\in E$ is said to be {\em $k$-spanned} in $M$ 
if there exist $k$ disjoint sets spanning $e$ (including the trivial spanning set $\{e\}$).

Consider a class of matroids such that no element is $(k+1)$-spanned. 
Kotlar and Ziv \cite{KZ05} provided two cases 
for which Problem~\ref{prob:main} admits solutions. 
\begin{theorem}[Kotlar and Ziv \cite{KZ05}]\label{thm:KZ1}
Let $M_{1}=(E,\I_{1})$ and $M_{2}=(E,\I_{2})$ be two matroids with rank functions 
$r_{1}$ and $r_{2}$ and suppose $r_{1}(E)=r_{2}(E)=d$ and $|E|=k\cdot d$.
If no element of $E$ is $(k+1)$-spanned in $M_{1}$ or $M_{2}$, 
then $E$ can be partitioned into $k$ common bases.
\end{theorem}

\begin{theorem}[Kotlar and Ziv \cite{KZ05}]\label{thm:KZ2}
Let $M_{1}=(E,\I_{1})$ and $M_{2}=(E,\I_{2})$ be two matroids.
If no element of $E$ is $3$-spanned in $M_{1}$ or $M_{2}$, 
then $E$ can be partitioned into two common independent sets.
\end{theorem}

Note that, in Theorems~\ref{thm:KZ1} and \ref{thm:KZ2},
the condition $E\in \I_{1}^{k}\cap\I_{2}^{k}$ is not
explicitly assumed. 
However, 
it 
%%is implicitly assumed as shown by the following lemma,  
%%which 
can be easily proved by induction on $|E|$.
\begin{lemma}[Kotlar and Ziv \cite{KZ05}]\label{lem:single_matroid}
If no element of a matroid $M=(E,\I)$ is $(k+1)$-spanned, then $E\in \I^{k}$.
\end{lemma}
%By this lemma, we have $E\in \I_{1}^{k}\cap\I_{2}^{k}$ 
%if no element of $E$ is $(k+1)$-spanned either in $M_{1}$ or $M_{2}$.

We provide unified proofs for Theorems~\ref{thm:KZ1} and \ref{thm:KZ2} via the generalized-polymatroid approach, 
%%To apply Proposition~\ref{prop:g-matroid_approach},
by constructing integral paramodular pairs satisfying \eqref{EQpb} in Proposition~\ref{prop:g-matroid_approach}.
We first show that the cross-inequality condition, which is required for  paramodularity,
is equivalent to a seemingly weaker condition.
%%The following lemma holds for arbitrary pair of set functions.
\begin{lemma}\label{lem:local-cross-ineq}
A pair $(p,b)$ of set functions satisfies 
the cross inequality \eqref{eq:cross-ineq} for any $A, B\subseteq E$ if and only if it satisfies 
the following inequality for every pair of disjoint subsets
$\tilde{A},\tilde{B}\subseteq E$ and \changeF{every} element $e\in E\setminus(\tilde{A}\cup \tilde{B})$: 
\begin{equation}
b(\tilde{A}+e)-b(\tilde{A})\geq p(\tilde{B}+e)-p(\tilde{B}). \label{eq:cross-ineq2}
\end{equation}
\end{lemma}
\begin{proof}
The necessity is obvious, 
since \eqref{eq:cross-ineq2} is obtained by substituting $A=\tilde{A}+e$ and $B=\tilde{B}+e$ into \eqref{eq:cross-ineq}.
For sufficiency, we show \eqref{eq:cross-ineq} for arbitrary $A, B\subseteq E$ under the assumption of \eqref{eq:cross-ineq2}.
Let $A\cap B=\{e_{1},e_{2},\dots,e_{m}\}$ where $m=|A\cap B|$ and 
define $\tilde{A}_{\ell}=(A\setminus B)\cup \{e_{1},e_{2},\dots,e_{\ell-1}\}$ and 
$\tilde{B}_{\ell}=(B\setminus A)\cup \{e_{\ell+1},e_{\ell+2},\dots,e_{m}\}$ for each $\ell\in\{1,2,\dots,m\}$. 
Then $\tilde{A}_{\ell}$ and $\tilde{B}_{\ell}$ are disjoint and $e_{\ell}\in E\setminus(\tilde{A}_{\ell}\cup \tilde{B}_{\ell})$,
and hence we have $b(\tilde{A}_{\ell}+e_{\ell})-b(\tilde{A}_{\ell})\geq p(\tilde{B}_{\ell}+e_{\ell})-p(\tilde{B}_{\ell})$ for $\ell=1,2,\dots, m$.
As we have $\tilde{A}_{1}=A\setminus B$, $\tilde{A}_{m}+e_{m}=A$, $\tilde{B}_{1}+e_{1}=B$, and $\tilde{B}_{m}=B\setminus A$, 
it follows that 
\[b(A)-b(A\setminus B)=\sum_{\ell=1}^{m}b(\tilde{A}_{\ell}+e_{\ell})-b(\tilde{A}_{\ell})
\geq \sum_{\ell=1}^{m}p(\tilde{B}_{\ell}+e_{\ell})-p(\tilde{B}_{\ell})=p(B)-p(B\setminus A).\]
Thus, $A$ and $B$ satisfy the cross inequality \eqref{eq:cross-ineq}. 
\end{proof}

Lemma \ref{lem:local-cross-ineq} states that, 
the range of the subsets $A,B \subseteq E$ in 
the cross inequality \eqref{eq:cross-ineq} can be narrowed so that $|A \cap B|=1$. 
\change{The above proof argument is the same as the standard proof argument for characterizing submodularity by the \emph{local submodularity} (see, e.g.,\ \cite[Theorem 44.1]{Schr03}) or 
the \emph{diminishing return property} (see, e.g.,\ \cite{SY15,SY16}). }
We also remark that the submodularity of $b$ and the supermodularity of $p$ are 
not assumed in Lemma \ref{lem:local-cross-ineq}. 

Now an integral paramodular pair satisfying the condition in Proposition \ref{prop:g-matroid_approach} is constructed as follows. 

\begin{lemma}\label{lem:g-polymatroid-approach-k+1}
Let $M=(E,\I)$ be a matroid with rank function $r:2^{E}\to \Zp$.
For a positive integer $k$, suppose that no element is $(k+1)$-spanned in $M$.
Define $p:2^{E}\to \Z$ and $b:2^{E}\to \Z$ by
\begin{alignat*}{2}
p(A)&=|A|-r^{k-1}(A)\quad&(A\subseteq E),\\
b(A)&=r(A)               &(A\subseteq E).
\end{alignat*}
Then  $(p,b)$ is an integral paramodular pair such that
$\F(p,b)=\set{X\subseteq E|X\in \I,~E\setminus X\in \I^{k-1}}$. 
\end{lemma}
\begin{proof}
It directly follows from the definitions of $p$ and $b$ that 
(i) $p(\emptyset)=b(\emptyset)=0$, 
(ii) $p$ is supermodular, $b$ is submodular.
Then, to prove that $(p,b)$ is paramodular, it remains to show 
the cross inequality \eqref{eq:cross-ineq} for any $A,B\subseteq E$.
By Lemma~\ref{lem:local-cross-ineq}, it suffices to show \eqref{eq:cross-ineq2} 
for any disjoint $\tilde{A},\tilde{B}\subseteq E$ and \changeF{any} element $e\in E\setminus(\tilde{A}\cup \tilde{B})$,
where \eqref{eq:cross-ineq2} is rephrased as follows by the definitions of $p$ and $b$:
\[\left(r(\tilde{A}+e)-r(\tilde{A})\right)+\left(r^{k-1}(\tilde{B}+e)-r^{k-1}(\tilde{B})\right)\geq 1.\]
Take a maximal independent set $X$ of $M$ subject to $X\subseteq \tilde{A}$ and
a maximal independent set $Y$ of $M^{k-1}$ subject to $Y\subseteq \tilde{B}$.
%Then $r(\tilde{A})=|X|$ and $r^{k-1}(\tilde{B})=|Y|$.
It is sufficient to show 
$X+e\in \I$ or $Y+e\in \I^{k-1}$, because 
they respectively imply $r(\tilde{A}+e)\geq r(\tilde{A})+1$ or
$r^{k-1}(\tilde{B}+e)\geq r^{k-1}(\tilde{B})+1$.
Note that $Y\in \I^{k-1}$ can be partitioned into $k-1$ 
independent sets $Y_{1},Y_{2},\dots,Y_{k-1}\in \I$. Also, 
$k+1$ subsets $\{e\}, X, Y_{1}, Y_{2},\dots,Y_{k-1}$ are all disjoint.
Because no element is $(k+1)$-spanned in $M$, it follows that $e$ is not spanned by at least one of $X, Y_{1}, Y_{2},\dots,Y_{k-1}$.
Note that $Y_{j}+e\in \I$ for some $j$ implies $Y+e\in \I^{k-1}$ by the definition of $\I^{k-1}$.
We then have $X+e\in \I$ or $Y+e\in \I^{k-1}$. Thus, the paramodularity of $(p,b)$ is proved.

We next show $\F(p,b)=\set{X\subseteq E|X\in \I,~E\setminus X\in \I^{k-1}}$.
For any $X\subseteq E$, the condition 
$\forall A\subseteq E: |X\cap A|\geq p(A)=|A|-r^{k-1}(A)$ is equivalent to 
$\forall A\subseteq E: |(E\setminus X)\cap A|\leq r^{k-1}(A)$,
and hence equivalent to $E\setminus X\in \I^{k-1}$.
Also, $\forall A\subseteq E: |X\cap A|\leq b(A)=r(A)$ is equivalent to $X\in \I$. 
Thus we have $X\in \F(p,b)$ if and only if $X\in \I$ and $E\setminus X\in \I^{k-1}$ hold. 
\end{proof}

Combining Lemmas~\ref{lem:single_matroid}, \ref{lem:g-polymatroid-approach-k+1} and Proposition~\ref{prop:g-matroid_approach} 
yields the following proposition.
\begin{proposition}\label{prop:induction_step}
Let $M_{1}=(E,\I_{1})$ and $M_{2}=(E,\I_{2})$ be two matroids.
If no element of $E$ is $(k+1)$-spanned in $M_{1}$ or $M_{2}$, 
%%\changeF{If every element of $E$ is $(k+1)$-spanned in neither $M_{1}$ nor $M_{2}$}, 
then there exists a subset $X\subseteq E$ such that $X\in \I_{1}\cap\I_{2}$ and $E\setminus X\in \I_{1}^{k-1}\cap \I_{2}^{k-1}$.
\end{proposition}

Using this proposition, we can provide unified proofs for Theorems~\ref{thm:KZ1} and \ref{thm:KZ2}.

\begin{proof}[Proof of Theorem~\ref{thm:KZ1}]
By Proposition~\ref{prop:induction_step}, there is $X\in \I_{1}\cap\I_{2}$ with $E\setminus X\in \I_{1}^{k-1}\cap \I_{2}^{k-1}$.
Because $r_{1}(E)=r_{2}(E)=d$ and $|E|=k\cdot d$, 
the subsets $X$ and  $E\setminus X$
should be common bases of $(M_{1}, M_{2})$ and $(M_{1}^{k-1}, M_{2}^{k-1})$, respectively.
For each matroid $M_{i}$ ($i=1,2$), since every element in $E\setminus X$ is spanned by $X$ 
but not $(k+1)$-spanned,
we see that no element in $E\setminus X$ is $k$-spanned in $M_{i}|(E\setminus X)$.
Thus, $X\in \I_{1}\cap \I_{2}$ and restrictions $M_{1}|(E\setminus X)$ and $M_{2}|(E\setminus X)$ satisfy 
the assumption of Theorem~\ref{thm:KZ1} with $k$ replaced by $k-1$. 
By induction, $E\setminus X$ can be partitioned into
$k-1$ common independent sets. Thus, the proof is completed
\end{proof}

\begin{proof}[Proof of Theorem~\ref{thm:KZ2}]
By just applying Proposition~\ref{prop:induction_step} with $k=2$, we obtain
a common independent set $X\in \I_{1}\cap\I_{2}$ satisfying $E\setminus X\in \I_{1}\cap \I_{2}$.
Thus, the proof is completed
\end{proof}

The original proofs for Theorems \ref{thm:KZ1} and \ref{thm:KZ2} \cite{KZ05} have no apparent relation. 
For these two theorems, 
we have shown unified proofs by our generalized-polymatroid approach. 
This offers a new understanding of  
the conditions in Theorems \ref{thm:KZ1} and \ref{thm:KZ2}: 
they are nothing other than conditions under which  
our induction method works.

\subsection{Intersection of a Laminar Matroid and a Matroid without $(k+1)$-Spanned Elements}
%%\subsection{Intersection of a Laminar Matroid and a Kotlar--Ziv Matroid}
\label{sec:new}

As mentioned in Remark \ref{REMindependent}, 
our g-polymatroid approach does not require 
the two matroids to be in the same matroid class, 
and thus 
can deal with 
an arbitrary pair of matroids which have appeared in this section. 
That is, 
we can obtain a solution of Problem~\ref{prob:main} for a new class of matroid intersection, 
i.e.,\ 
the intersection of 
a laminar matroid 
and 
a matroid without $(k+1)$-spanned elements. 
The following theorems can be immediately derived from 
combining the proofs of Theorems~\ref{thm:laminar-partition}, \ref{thm:KZ1}, and \ref{thm:KZ2}. 
%%which show that Problem~\ref{prob:main} is solvable if 
%%one matroid is laminar and 
%%%%the other has no $(k+1)$-spanned element and satisfies a rank condition. 
%%the other matroid belongs to one of the two classes of Kotlar and Ziv \cite{KZ05}. 

\begin{theorem}\label{thm:lam+KZ1}
Let 
$k$ be a positive integer,
$M_{1}=(E,\I_{1})$ be a laminar matroid such that \changeF{$E \in \I_1^k$}, 
and 
$M_{2}=(E,\I_{2})$ be a matroid with rank function $r_{2}$ such that 
$|E|=k\cdot r_{2}(E)$ and no element is $(k+1)$-spanned in $M_{2}$.
Then, 
%%$E\in \I_{1}^{k}\cap \I_{2}^{k}$, 
there exists a
partition $\{X_{1},X_{2},\dots,X_{k}\}$ of $E$ such that $X_{j}\in \I_{1}\cap \I_{2}$ for each $j=1,2,\dots,k$.
\end{theorem}

\begin{theorem}\label{thm:lam+KZ2}
Let $M_{1}=(E,\I_{1})$ be a laminar matroid such that 
$E \in \I_1^2$ and $M_{2}=(E,\I_{2})$ be a matroid \changeF{in which no element is $3$-spanned in $M_{2}$.}
%%If $E\in \I_{1}^{k}\cap \I_{2}^{k}$, 
Then, 
there exists a
partition $\{X_{1},X_{2}\}$ of $E$ such that $X_{1},X_2\in \I_{1}\cap \I_{2}$.
\end{theorem}

\subsection{Time Complexity}\label{sec:complexity}

The proof for Proposition \ref{prop:g-matroid_approach} implies a polynomial-time algorithm to solve 
Problem~\ref{prob:main} for matroids mentioned above.
Here we 
discuss the time complexity of the algorithm. 

Suppose that $M_{1}=(E,\I_{1})$ and $M_{2}=(E,\I_{2})$ satisfy the assumptions in Proposition~\ref{prop:g-matroid_approach}, 
and membership oracles of $\I_{1}$ and $\I_{2}$ are provided.
Let $\J_i:=\set{X| X\in \I_i, E\setminus X\in \I_i^{k-1}}$ for each $i=1,2$. 
%%Then our objective is to find a subset $X\in \J_{1}\cap \J_{2}$, 
%%whose existence is guaranteed by Proposition~\ref{prop:g-matroid_approach}.
We solve Problem \ref{prob:main} by $k-1$ iterations of finding $X \in \J_1 \cap \J_2$, 
which can be done efficiently in the following manner. 

By the assumption, $\J_i=\F(p_{i},b_{i})$ holds for some intersecting-paramodular pair $(p_{i},b_{i})$,
and hence $(E,\J_i)$ is a generalized matroid \cite{FT88, Tardos85} (see Lemma~\ref{lem:g-matroid} below).
%%Then, it is known \cite{Tardos85} that $\J_i$ is a projection of the base family of another matroid, 
\changeF{Then, it is known that $\J_i$ is a projection of the base family of another matroid \cite{Tardos85} (see also \cite{Fuj84,Fuj05})}, 
which is defined as follows. 
Let $m_{\rm max}$ and $m_{\rm min}$ be the maximum and minimum size of a subset in $\J_1\cup \J_2$ respectively, 
and let $U$ be a set disjoint from $E$ with $|U|=m_{\rm max}-m_{\rm min}$.
As shown in \cite[Theorem~2.9]{Tardos85},  the family 
\[\tilde{\B}_i=\set{X\cup V|X\in \J_i, ~V\subseteq U, ~|X\cup V|=m_{\rm max}}\]
forms the base family of a matroid on $E\cup U$, which we denote by $\tilde{M}_i=(E\cup U, \tilde{\I}_i)$.  
Then we have $\J_i=\set{E\cap B|B\in \tilde{\B}_i}$ for each $i=1,2$,
and hence $\J_1\cap \J_2=\set{E\cap B|B\in \tilde{\B}_1\cap \tilde{\B}_2}$. 

Therefore, finding $X\in \J_1\cap \J_2$ is reduced to
finding a common base of $\tilde{M}_1$ and $\tilde{M}_2$.
This can be done by a standard matroid intersection algorithm \cite{Edmonds70}, 
if membership oracles of $\tilde{\I}_1$ and $\tilde{\I}_2$ are available.
Such oracles can be efficiently implemented as follows. 

Since $\tilde{\I}_i$ consists of all subsets of the bases in $\tilde{\B}_i$,  
it follows from the definitions of $\tilde{\B}_i$ and $\J_i$ that 
\[\tilde{\I}_i=\set{X'\cup V|\exists X\in \I_i: X'\subseteq X,~|X\cup V|\leq m_{\rm max},~E\setminus X\in \I_i^{k-1} }.\] 
Then, we have $X'\cup V\in \tilde{\I}_i$ if and only if $E\setminus X'$ is partitionable into $Z, Y_1,Y_2,\dots,Y_{k-1}$ such that 
\[Z\in \I^{\ast}_i:=\set{Z' \subseteq E \setminus X'|Z'\cup X'\in \I_i, ~|Z'\cup X'\cup V|\leq m_{\rm max}}\] and $Y_j\in \I_i~(j=1,\dots, k-1)$.
Note that $\I_i^{\ast}$ is also the independent set family of a matroid (as it is obtained from $\I_i$ by contraction and truncation).
Therefore, by the matroid partition algorithm \cite{Edmonds65}, we can decide whether $X'\cup V\in \tilde{\I}_i$ using the oracle of $\I_i$.

Here we show a complexity analysis of the above algorithm with the implementations of matroid intersection and matroid partition algorithms by Cunningham \cite{Cun86}. 
%%Now the total time complexity is analyzed as follows. 
Let 
$n=|E|$, 
$r = \max\{r_1(E),r_2(E)\}$, 
where $r_i$ denotes the rank function of $M_i$ for $i=1,2$, 
and 
let $\tau$ denote the time for the membership oracles of $\I_1$ and $\I_2$. 
Overall, 
we solve $k-1$ instances of the matroid intersection problem defined by $\tilde{M}_1$ and $\tilde{M}_2$. 
Each instance can be solved in $O(r^{1.5}n Q)$ time \cite{Cun86}, 
where $Q$ is the time for an independence test in $\tilde{M}_1$ and $\tilde{M}_2$. 
Testing independence in $\tilde{M}_1$ and $\tilde{M}_2$ can be done in $O(n^{2.5} \tau)$ time \cite{Cun86}. 
Therefore, the total time complexity is $O(k r^{1.5} n^{3.5} \tau)$.

\section{Example Incompatible with the G-polymatroid Approach}
\label{sec:incompatible}
%%Besides the classes of matroids in Sections~\ref{sec:laminar} and \ref{sec:k+1-spanned}, 
%%which can be proved via the g-polymatroid approach, 
As mentioned before, 
the class of strongly base orderable matroids admits a solution for Problem~\ref{prob:main} \cite{DM76}, 
and 
our g-polymatroid approach can deal with laminar matroids, 
a special class of strongly base oderable matroids. 
Hence we expect that the g-polymatroid approach can be applied to strongly base orderable matroids. 
In this subsection, 
however, we \change{show the limits of our method} 
by constructing an example of a transversal matroid, 
another simple special case of a strongly base orderable matroid, 
which admits 
no intersecting-paramodular pair required in Proposition~\ref{prop:g-matroid_approach}. 

%%\begin{definition}[Strongly base orderable matroid]
%%A matroid is {\em strongly base orderable} if for each bases $B_{1}$, $B_{2}$
%%there exists a bijection $\pi:B_{1}\to B_{2}$ such that for each subset $X$ of $B_{1}$
%%the set $\pi(X)\cup (B_{1}\setminus X)$ is a base again.
%%\end{definition}

For a bipartite graph $G=(E,F;A)$ with color classes $E$, $F$ and edge set $A$,
let $\I$ be a family of subsets $X$ of $E$ such that $G$ has a matching 
that is incident to all elements in $X$.
Then it is known that $(E,\I)$ is a matroid \cite{EF65}, 
which we call the {\em transversal matroid} induced from $G$.
%%The rank function $r$ is given by
%%$r(D)=|\Gamma(D)|$ for any $D\subseteq E$ where $\Gamma(D)$ is the set of nodes 
%%in $F$ that is adjacent to some node in $D$.

%%\begin{definition}[Transversal matroid]
%%A matroid is {\em transversal} if it is induced by some bipartite graph.
%%\end{definition}

%~\\\noindent\change{\bf SIGNIFICANT CHANGE $\downarrow$}

\begin{example}\label{ex:1}
\change{
Let $E=\{e_{1},e_{2},e_{3}, e'_{1},e'_{2}, e'_{3}\}$ and
$G$ be a bipartite graph with color classes $E$ and $F$. 
Figure~\ref{fig1} depicts $G$, where the white and black nodes represent $E$ and $F$, respectively.
Let $M=(E, \I)$ be the matroid induced by $G$.
Observe that there exists a matching between $\{e'_{1},e'_{2},e'_{3}\}$ and $F$,
%%which is maximal in $G$. 
%%Thus, 
implying that 
the rank of $M$ is three.}
\begin{figure}[htbp]
\begin{center}
   \includegraphics[width=35mm]{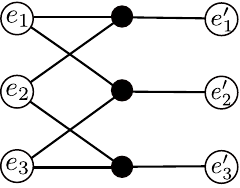}
\caption{A bipartite graph that induces a transversal matroid  $M=(E,\I)$ in Example~\ref{ex:1}}
\label{fig1}
\end{center}
\end{figure}
\end{example}

We show that transversal matroids are incompatible with
our g-polymatroid approach by proving that 
the transversal matroid $M$ in Example \ref{ex:1} admits 
no intersecting-paramodular pair $(p,b)$ satisfying 
$\F(p,b)=\set{X\subseteq E|X\in \I,~E\setminus X\in \I^{k-1}}$, 
i.e.,\ 
we cannot use Proposition~\ref{prop:g-matroid_approach} for $M$.
For this purpose, we prepare the following fact (see e.g., \cite{Frankbook, Mbook, MS99}).

\begin{lemma}\label{lem:g-matroid}
For any integral intersecting-paramodular pair $(p,b)$, 
if $\J:=\F(p,b)\neq \emptyset$, then $(E,\J)$ is a {\em generalized matroid}, i.e., 
$\J$ satisfies the following axioms:
\begin{itemize}
\setlength{\parskip}{0.0cm}
\setlength{\itemsep}{0.0mm}
\item[\rm (J1)] If $X, Y \in\J$ and $e \in Y\setminus X$, then $X +e \in \J$ or $\exists e'\in X\setminus Y: X+e-e'\in \J$.
\item[\rm (J2)] If $X, Y \in\J$ and $e \in Y\setminus X$, then $Y-e \in \J$ or $\exists e'\in X\setminus Y: Y-e+e'\in \J$.
\end{itemize}
\end{lemma}

%%We now give an example which admits  
%%no intersecting-paramodular pair required in Proposition~\ref{prop:g-matroid_approach}.
Now the incompatibility of the transversal matroid $M$ in Example \ref{ex:1} is established by the following proposition.

\begin{proposition}
For the transversal matroid $M=(E,\I)$ given in Example~\ref{ex:1},
there is no integral intersecting-paramodular pair $(p,b)$ satisfying
$\F(p,b)=\set{X\subseteq E|X\in \I,~E\setminus X\in \I^{k-1}}$ with $k=2$, while $E$ can be partitioned into two independent sets in $M$.
\end{proposition}
\begin{proof}
The latter statement is obvious: $E$ can be partitioned into two bases 
$\{e_{1}, e_{2},e_{3}\}$ and $\{e'_{1},e'_{2},e'_{3}\}$. 

Suppose to the contrary that $\F(p,b)=\set{X\subseteq E|X\in \I,~E\setminus X\in \I}$ holds
for some integral intersecting-paramodular  pair $(p,b)$.
By Lemma~\ref{lem:g-matroid}, then $\J:=\F(p,b)=\set{X\subseteq E|X\in \I,~\overline{X}\in \I}$ 
satisfies (J1) and (J2), where $\overline{X}$ denotes $E\setminus X$.

Let $X:=\{e_{1}, e'_{2}, e'_{3}\}$ and 
$Y:=\{e'_{1}, e'_{2}, e_{3}\}$. 
We can observe $X, \overline{X}, Y, \overline{Y}\in \I$, 
and hence $X,Y\in \J$.
Apply (J1) to $X\in \J$ and $e_{3}\in Y\setminus X$. As $X+e_{3}\not \in \J$ and $X\setminus Y=\{e_{1},e'_{3}\}$, 
we must have $X+e_{3}-e_{1}\in \J$ or $X+e_{3}-e'_{3}\in \J$.
However, 
it holds that 
$X+e_{3}-e_{1}=\{e'_{2}, e_{3}, e'_{3}\}\not\in \I$,
implying  $X+e_{3}-e_{1}\not\in \J$, 
and 
it also holds that $\overline{X+e_{3}-e'_{3}}=\{e'_{1}, e_{2}, e'_{3}\}\not\in \I$, 
implying $X+e_{3}-e'_{3}\not\in \J$, a contradiction.
\end{proof}

\section*{Acknowledgments}
The authors would like to thank Satoru Fujishige and Kazuo Murota for valuable comments. 
The authors are also obliged to anonymous referees for helpful comments, 
in particular on a smaller example of a transversal matroid. 
The first author is partially supported by 
JST CREST Grant Number JPMJCR1402,  %%Katoh
JSPS KAKENHI Grant Numbers 
JP16K16012, %%Takazawa
JP26280004, %%Murota
Japan.
The second author is supported by
JST CREST Grant Number JPMJCR14D2, %%Iwata
JSPS KAKENHI Grant Number JP18K18004, %%Yokoi
Japan.
\bibliographystyle{myplain}%           BibTeX を使う場合
\bibliography{myrefs}% BibTe\X を使う場合
\end{document}